\documentclass[12pt]{amsart}

\usepackage{hyperref}
\hypersetup{nesting=true,debug=true,naturalnames=true}
\usepackage{graphicx,amssymb,upref}

\usepackage{amsmath,amsthm}
\usepackage{amsfonts}
\usepackage{latexsym}
\usepackage{amsbsy}
\usepackage{amssymb,mathscinet}
\usepackage{tikz-cd}
\usepackage{mathtools}

\numberwithin{equation}{section}

\usepackage{amsfonts,amssymb,amscd,amsmath,enumerate,verbatim,calc,enumerate}
\theoremstyle{plain}
\newtheorem{theorem}{Theorem}[section]

\newtheorem{lemma}[theorem]{Lemma}

\newtheorem{observation}[theorem]{Observation}

\newtheorem{definition}[theorem]{Definition}

\newcommand{\be}{\mathbb E}
\newcommand{\bn}{\mathbb N}
\newcommand{\Nk}{\bn_0^k}
\newcommand{\ot}{\otimes}
\newcommand {\id} {{\textrm{id}}}
\newcommand{\wt}{\widetilde}
\newcommand{\wT}{\wt{T}}

\begin{document}
	
\title[Ando type dilation for completely contractive covariant reps] 
{Ando type dilation for completely contractive covariant representations}

	\date{\today}
		\author[Rohilla]{Azad Rohilla}
	\address{Department of Mathematics, Sanskaram University, Jhajjar, India}
	\email{18pmt005@lnmiit.ac.in}
	\author[Saini]{Dimple Saini \textsuperscript{*}}
	\address{Department of Mathematics, Gautam Buddha University, Greater Noida, India}
	\email{dimple92.saini@gmail.com}

\begin{abstract}
	The celebrated theorem of Ando says that a pair of commuting contraction on a Hilbert space always dilates to an isometric dilation. Building on this, Solel proved that every completely contractive covariant representation of a product system $\be$ over $\mathbb N^2_0$ dilates to an isometric covariant
	representation of $\be$ over $\mathbb N^2_0.$ This result generalizes Ando’s theorem in the special case where the $C^{*}$-algebra $\mathcal{M}=\mathbb{C}$ and $\mathbb{E}(\mathbf{n})=\mathbb{C}$ for all $\mathbf{n}\in \mathbb N^2_0.$ This naturally leads to the question: What class of isometric covariant representations of a product system $\be$ over $\mathbb N^2_0$ are sufficient to
	serve as dilations for completely contractive covariant representations, as established in Solel’s dilation framework? The central objective of this article is to investigate such isometric covariant representations that serve as dilations of completely contractive covariant representations.
\end{abstract}



\maketitle


\section{Introduction}
The starting point of dilation theory is Sz.-Nagy theorem \cite{NFS}, which states that every contraction on a Hilbert space dilates to an isometry on a larger Hilbert space. This result laid the foundation for a deeper understanding of operator theory, providing a powerful tool for studying contractions through their isometric extensions. Ando \cite{AT63}, generalizing Sz.-Nagy’s dilation, constructed isometric dilations for pairs of commuting contractions. That is, for every pair of commuting contractions $(T_1,T_2)$ on a Hilbert space $\mathcal{ H},$ then there exists a pair of commuting isometries $(V_1,V_2)$ on a Hilbert space $\mathcal{K} \supseteq \mathcal{H}$ such that  
$$T_1^{n_1}T_2^{n_2}=P_{\mathcal{ H}}V_1^{n_1}V_2^{n_2}|_{\mathcal{ H}} \quad \quad \mbox{for all} \quad \quad (n_1,n_2)\in \mathbb{Z}_{+}^2,$$ where $P_{\mathcal{ H}}$
denotes the orthogonal projection onto ${\mathcal{ H}}.$
Ando’s theorem confirms that every pair of commuting contractions can be dilated to a pair of commuting isometries, providing a complete solution in the two-variable case. Parrott \cite{SP70} gave a counterexample showing that for $n\ge 3,$ $n$-tuples of commuting contractions do not generally admit isometric dilations. This striking limitation highlights the intrinsic difficulties of multivariable dilation theory and motivates a more nuanced investigation into the specific conditions under which such dilations might still be possible.

In recent years, a growing body of literature has explored the interplay between dilation and positivity, revealing that certain classes of operator tuples under specific positivity assumptions can indeed possess isometric dilations. See \cite{BKHJ19,CV93,BJ17,GS97,MV93} for
more details in the polydisc setup. The main motivation behind our interest comes from a recent work by Das and Sarkar \cite{BJ17}, obtained an explicit isometric dilation of pair of commuting contractions in the following theorem:

\begin{theorem}
	Let $(T_1, T_2)$ be a pair of commuting contractions on a Hilbert space $\mathcal{ H}.$ Let $T_1$ be pure and dim $\mathcal{D}_{T_j} <\infty, j=1,2.$ Then there exist an isometry $\Pi :\mathcal{ H}\to H^2_{\mathcal{D}_{T_1}}(\mathbb{D})$ and an inner function $\psi\in H^{\infty}_{\mathcal{B}(\mathcal{D}_{T_1})}(\mathbb{D})$ such that $$\Pi T_1^*=M_z^*\Pi \quad \quad \mbox{and}\quad \quad  \Pi T_2^*=M_{\psi}^*\Pi,$$ where $\mathcal{D}_{T_j}=\overline{ran}(I_{\mathcal{ H}} -T_jT^*_j)$ for $j=1,2.$
\end{theorem}

Operator and function theories are closely intertwined and have long been regarded as well-established fields with significant connections to a wide range of disciplines. These areas continue to advance, both by addressing traditional problems and by extending classical ideas to more generalized frameworks. One such extension is the concept of Hilbert $C^*$-modules, developed through $C^*$-correspondences and covariant representations. 

Cuntz \cite{C77} studied a $C^*$-algebra, known as Cuntz algebra, generated by isometries $V_1,\dots,V_n$  $(n\ge 2)$ on a Hilbert space with $V_i$'s have orthogonal range. Wold decomposition for two isometries with orthogonal range was first introduced by Frazho \cite{F84}. Popescu \cite{POP89} extended this decomposition to the case of an infinite sequence of isometries with orthogonal final spaces, and explored a minimal isometric dilation for an infinite sequence of noncommuting contractive operators on a Hilbert space. Pimsner \cite{P97} generalized the construction of Cuntz algebras with the help of the notion of isometric representations of $C^*$-correspondences. Muhly and Solel \cite{MR1648483} discussed the Wold decomposition for the isometric covariant representations of $C^*$-correspondences based on \cite{POP89}. 

Muhly and Solel \cite{MS98} presented isometric dilation for completely contractive covariant representations of $C^*$-correspondences. Solel \cite{S06} proved that every completely contractive covariant representation of product system of $C^*$-correspondences over the semigroup $\mathbb{N}_0^2$ dilates to an isometric representation. Solel \cite{S08} studied completely contractive representations of product system of $C^*$-correspondences over $\mathbb{N}_0^k,$ and gave a necessary and sufficient condition for such a representation to have a regular isometric dilation. In \cite{S009} Skalski established some sufficient conditions for the existence of a not necessarily regular isometric dilation of a completely contractive representation of the product system and discussed the difference between regular and $*$-regular dilations. The main purpose of this paper is to study such isometric covariant representations that serve as dilations of completely contractive covariant representations.

\subsection{Notations and Preliminaries}

We begin this section by recalling few concepts from  \cite{La95,MS98,MR0355613,RTV,DTV2}. Let $E$ be a Hilbert $C^*$-module over a $C^*$-algebra $\mathcal M.$ By $B^a(E)$ we denote the $C^*$-algebra of all
adjointable operators on $E$. We say that the module $E$ is a {\it
	$C^*$-correspondence over $\mathcal M$} if there exists a left $\mathcal
M$-module structure through a non-zero $*$-homomorphism
$\phi:\mathcal M\to B^a(E)$ in the following sense
\[
a\xi:=\phi(a)\xi \quad \quad (a\in\mathcal M, \xi\in E).
\]
In this article, each $*$-homomorphism considered is
essential, this means that, the closed linear span of
$\phi(\mathcal M)E$ equals $E.$ If $F$ is another $C^*$-correspondence over $\mathcal M,$ then we may consider the notion of
tensor product $F\bigotimes_{\phi} E$ (cf. \cite{La95}) which satisfy
\[
(\eta_1 a)\otimes \xi_1=\eta_1\otimes \phi(a)\xi_1,
\]
\[
\langle\eta_1\otimes\xi_1,\eta_2\otimes\xi_2\rangle=\langle\xi_1,\phi(\langle\eta_1,\eta_2\rangle)\xi_2\rangle
\]
for every $\eta_1,\eta_2\in F,$ $\xi_1,\xi_2\in E$ and $a\in\mathcal
M.$ In this paper, we adopt the following notations: $\mathcal{H}$ denotes a Hilbert space, $E$ represents a $C^*$-correspondence over $\mathcal{M}$, and $B(\mathcal{H})$ refers to the algebra of all bounded linear operators on $\mathcal{H}.$ The notion of completely contractive covariant representation plays a crucial role in this paper.

\begin{definition}
	Let $\sigma:\mathcal M\to B(\mathcal H)$ be a representation and $T: E\to B(\mathcal H)$ be a linear map. Then the tuple $(\sigma,T)$ is said to be a {\rm covariant representation} of $E$ on $\mathcal H$ if
	\[
	T(a\xi a')=\sigma(a)T(\xi)\sigma(a') \quad \quad (\xi\in E,
	a,a'\in\mathcal M).
	\]
	We say that the covariant representation $(\sigma,T)$ is {\rm completely
		bounded (respectively, completely contractive)} if $T$ is completely bounded (respectively, completely contractive ) as a linear map on $E$ (endowed with the usual operator space structure by viewing as a left corner in the linking algebra $\mathcal{L}_E$ of $E$). Moreover, $(\sigma,T)$ is called {\rm isometric} if
	\[
	T(\xi)^*T(\zeta)=\sigma(\langle \xi,\zeta\rangle) \quad \quad
	(\xi,\zeta\in E).
	\]
	
\end{definition}

Muhly and Solel gave the following key lemma in \cite{MS98} which is useful to classify the completely bounded covariant representations of a $C^*$-correspondence:
\begin{lemma}\label{DSAZ1}
	The map $(\sigma,T)\mapsto \widetilde T$ provides a bijection
	between the collection of all completely bounded covariant (respectively, completely contractive covariant)
	representations $(\sigma,T)$ of $E$ on $\mathcal H$ and the
	collection of all bounded (respectively, contractive) linear maps
	$\widetilde{T}:~\mbox{$E\bigotimes_{\sigma} \mathcal H\to \mathcal
		H$}$ defined by
	\[
	\widetilde{T}(\xi\otimes h):=T(\xi)h \quad \quad (\xi\in E,
	h\in\mathcal H),
	\]
	such that $\widetilde{T}(\phi(a)\otimes I_{\mathcal
		H})=\sigma(a)\widetilde{T}$, $a\in\mathcal M$. Moreover, $\widetilde T$ is isometry (respectively, co-isometry) if and only if $(\sigma,T)$ is isometric (respectively, co-isometric).
\end{lemma}

\begin{definition}	
	Suppose that $(\sigma,T)$ is a completely bounded covariant  representation of ${E}$ on $\mathcal{H}.$ A closed  subspace $\mathcal{W}$ of $\mathcal{H}$ is said to be $(\sigma,T)$-{\rm invariant} $(resp. (\sigma,T)$-{\rm reducing}) (cf. \cite{DTV2}) if it is $\sigma(\mathcal{M})$-invariant and (resp. both $\mathcal{W},\mathcal{W}^{\perp}$) is invariant by each operator  $T(\xi)$ for all $\xi \in E.$ The restriction provides a new representation $(\sigma , T)|_{\mathcal{W}}$ of $E$ on $\mathcal{W}.$
\end{definition}

A completely bounded covariant representation $(\sigma, T)$ of $E$ on $\mathcal{H}$ is said to be {\it completely non unitary} if there doesn’t exists a non-zero reducing subspace $\mathcal{W}$ of $\mathcal{H}$ such that $(\sigma, T)|_{\mathcal{W}}$ is isometric as well as co-isometric representation.

For each $n\in \mathbb{N},$
$E^{\otimes n}: =E\otimes_{\phi} \cdots \otimes_{\phi}E$ ($n$ times) ($E^{\otimes 0}=\mathcal{M}$) is a $C^*$-correspondence over the $C^*$-algebra $\mathcal M$, where the left action  of $\mathcal{M}$ on $E^{\otimes n}$  is given by  $$\phi^n(a)(\xi_1 \otimes \cdots \otimes \xi_n):=\phi(a)\xi_1\otimes \cdots \otimes\xi_n.$$ 
For  $n\in \mathbb{N},$ define $\widetilde{T}_n : E^{\otimes n}\otimes \mathcal{H} \to \mathcal{H}$ by $$\widetilde{T}_n (\xi_1 \otimes \dots \otimes \xi_n \otimes h) = T (\xi_1) \dots T(\xi_n) h, \quad \xi_i \in E, h \in \mathcal H.$$
The Fock module $\mathcal{F}(E)= \bigoplus_{n \geq 0}E^{\otimes n}$ is a $C^*$-correspondence over a $C^*$-algebra  $\mathcal M$, with left action  of $\mathcal{M}$ on $\mathcal{F}(E)$ is given by  $$\phi_{\infty}(a)(\oplus_{n \geq 0}\xi_n)=\oplus_{n \geq 0}a\cdot\xi_n , \:\: \xi_n \in E^{\otimes n}, a\in \mathcal{M}.$$
Let $\xi \in E$, we define the creation operator $T_{\xi}$ on $\mathcal{F}(E)$ by $$T_{\xi}(\eta):=\xi \otimes \eta, \:\: \eta \in E^{\otimes n}.$$
Let $\pi $ be a representation of $\mathcal{M}$ on $\mathcal{H}$. Define an isometric covariant representation $(\rho, S)$ of $E$ on a Hilbert space $\mathcal{F}(E)\otimes_{\pi}\mathcal{H}$ by
\begin{align*}
	\rho(a):&=\phi_{\infty}(a) \otimes I_{\mathcal{H}} \:\:, a \in \mathcal{M}&\\
	S(\xi):&=T_{\xi}\otimes I_{\mathcal{H}} ,\:\: \xi \in E,
\end{align*} such a representation $(\rho, S)$ is referred to as the induced representation
(cf. \cite{SZ08}) induced by $\pi.$ Let $(\sigma, T)$ and $(\psi, V)$  be two completely bounded covariant representations of $E$ on the Hilbert spaces $\mathcal{H}$ and $\mathcal{K},$ respectively. We say that $(\sigma, T)$ is {\it isomorphic} to $(\psi, V)$ if there exists a unitary operator $U: \mathcal{H} \to\mathcal{K}$ such that $U\sigma(a)=\psi(a)U$ and $UT(\xi)=V(\xi)U$ for all $\xi \in E, a \in \mathcal{M}.$  A closed subspace $\mathcal{W}\subseteq \mathcal{H}$ is called {\rm wandering subspace} for $(\sigma , T)$ if $\mathcal{W}$ is orthogonal to
$\widetilde{T}_n(E^{\ot n}\ot \mathcal{W})$ for all $n\in \mathbb{N}.$ We say that the wandering subspace $\mathcal{W}$ is {\rm generating} for $(\sigma ,T)$ 
if  $$\mathcal{H}= \bigvee_{n\ge 0}\widetilde{T}_n(E^{\ot n}\ot \mathcal{W}).$$

The following theorem is an analogue of \cite[Theorem 3.2]{NFS}.

\begin{theorem}\label{contra}
	Let $(\sigma, T)$ be a completely contractive covariant representation of $E$ on $\mathcal{H}.$ Then there exists a unique orthogonal decomposition of $\mathcal{H}=\mathcal{H}_{0}\oplus\mathcal{H}_{1}$ such that the following holds:\begin{enumerate}
		\item $\mathcal{H}_{0}$ and $\mathcal{H}_{1}$ reduce $(\sigma, T).$
		\item $(\sigma, T)|_{\mathcal{H}_{0}}$ is an isometric as well as co-isometric representation.
		\item $(\sigma, T)|_{\mathcal{H}_{1}}$ is completely non unitary.
	\end{enumerate}  
\end{theorem}
\begin{proof}
	We define \begin{align*}
		\mathcal{H}_{0}:&=\{h \in \mathcal{H}: \|(I_{E^{\ot n}}\ot \widetilde{T}_{m})(\xi_{n+m}\ot h)\|=\|\xi_{n+m}\ot h\| \quad \mbox{and}\quad  \\& \|(I_{E^{\ot n}}\ot \widetilde{T}_m^{*})(\xi_{n}\ot h)\|=\|\xi_{n}\ot h\| \quad\mbox{for}\quad n,m\in \mathbb{N}_0, \xi_{n+m}\in E^{\ot n+m}\}.
	\end{align*}
	It is easy to verify that $\mathcal{H}_{0}$ is  a closed subspace of $\mathcal{H}.$ Now we need to show that $\mathcal{H}_{0}$ reduces $(\sigma, T).$ Let $h\in \mathcal{H}_0,$ we want to prove that $\widetilde{T}(\xi\ot h)\in \mathcal{H}_{0}$ for all $\xi\in E.$ Note that \begin{align*}&
		\|(I_{E^{\ot n}}\ot \widetilde{T}_m)(\xi_{n+m}\ot \widetilde{T}(\xi\ot h))\|=\|(I_{E^{\ot n}}\ot \widetilde{T}_{m+1})(\xi_{n+m}\ot \xi\ot h)\|\\&=\|\xi_{n+m}\ot \xi\ot h\|=\|(I_{E^{\ot n+m}}\ot \widetilde{T})(\xi_{n+m}\ot \xi\ot h)\|=\|\xi_{n+m}\ot \widetilde{T}(\xi\ot h)\|.
	\end{align*} If $\|\widetilde{T}(\xi\ot h)\|=\|\xi\ot h\|,$ then $\widetilde{T}^{*}\widetilde{T}(\xi\ot h)=\xi\ot h$ we have  \begin{align*}&
		\|(I_{E^{\ot n}}\ot \widetilde{T}_m^{*})(\xi_{n}\ot \widetilde{T}(\xi\ot h))\|\\&=\|(I_{E^{\ot n}}\ot (I_{E} \ot \widetilde{T}_{m-1}^{*})\widetilde{T}^{*})(\xi_{n}\ot \widetilde{T}(\xi\ot h))\|\\&=\|\xi_{n}\ot (I_{E} \ot \widetilde{T}_{m-1}^{*})\widetilde{T}^{*}\widetilde{T}(\xi\ot h)\|=\|\xi_{n}\ot (I_{E} \ot \widetilde{T}_{m-1}^{*})(\xi\ot h)\|\\&=\|(I_{E^{\ot n+1}} \ot \widetilde{T}_{m-1}^{*})(\xi_{n}\ot \xi\ot h)\|=\|\xi_{n}\ot \xi\ot h\|\\&=\|(I_{E^{\ot n+1}} \ot \widetilde{T})(\xi_{n}\ot \xi\ot h)\|=\|\xi_{n}\ot \widetilde{T}(\xi\ot h)\|.
	\end{align*} This implies that $\widetilde{T}(\xi\ot h)\in \mathcal{H}_{0}.$ Also we observe that 
	\begin{align*}
		\|(I_{E^{\ot n}}\ot \widetilde{T}_{m})(\xi_{n+m-1}\ot \widetilde{T}^{*} h)\|=\|\xi_{n+m-1}\ot \widetilde{T}^{*}h\|,
	\end{align*} and \begin{align*}
		\|(I_{E^{\ot n}}\ot \widetilde{T}_m^{*})(\xi_{n-1}\ot \widetilde{T}^{*} h)\|=\|\xi_{n-1}\ot \widetilde{T}^{*} h\|, \quad \mbox{for all} \quad m,n\ge 1.
	\end{align*} This implies that $\mathcal{H}_{0}$ reduces $(\sigma, T).$ Also it is easy to see that $(\sigma, T)|_{\mathcal{H}_{0}}$ is an isometric as well as co-isometric representation. Since $\mathcal{H}_{1}=\mathcal{H}\ominus\mathcal{H}_{0},$ $(\sigma, T)|_{\mathcal{H}_{1}}$ is completely non-unitary. Indeed, let $\mathcal{K}$ be a subspace of $\mathcal{H}_{1}$ such that $(\sigma, T)|_{\mathcal{K}}$ is a unitary. It follows that $\mathcal{K}\subseteq {\mathcal{H}_{0}},$  which is a contradiction.

	Uniqueness: Let $\mathcal{H}=\mathcal{H}'_{0}\oplus\mathcal{H}'_{1}$ be an orthogonal decomposition of $\mathcal{H}$ such that $(\sigma, T)|_{\mathcal{H}'_{0}}$ is an isometric as well as co-isometric representation, then $\mathcal{H}'_{0}\subseteq \mathcal{H}_{0}.$ The spaces $\mathcal{H}'_{0}$ and $\mathcal{H}_{0}$ reduce $(\sigma, T),$ and hence $(\sigma, T)|_{\mathcal{H}_{0}\ominus\mathcal{H}'_{0}}$ is an isometric as well as co-isometric representation. Note that $\mathcal{H}_{0}\ominus\mathcal{H}'_{0}\subseteq \mathcal{H}\ominus\mathcal{H}'_{0}=\mathcal{H}'_{1}.$ Since $(\sigma, T)|_{\mathcal{H}'_{1}}$ is completely non unitary, $\mathcal{H}_{0}\ominus\mathcal{H}'_{0}=\{0\}.$ Therefore $\mathcal{H}_{0}=\mathcal{H}'_{0},$ and hence the orthogonal decomposition is unique.
\end{proof}

Let $(\sigma, T)$ be a completely contractive covariant representation of $E$ on $\mathcal{H},$ that is, $I_{\mathcal{H}}-\wT\wT^*\ge 0.$ Define the positive operator $\Delta_*(T):=(I-\widetilde{T}{\widetilde{T}}^*)^{\frac{1}{2}}$ and defect space $\mathcal{D}_{*,T}:=\overline{Im\Delta_*(T)}$. Since $\wT\wT^*\sigma(a)=\wT(\phi(a)\ot I_{\mathcal{H}})\wT^*=\sigma(a)\wT\wT^*$ for all $a\in \mathcal{M},$ we get  $\Delta_*(T) \in \sigma(\mathcal{M})'.$ Thus $\mathcal{D}_{*,T}$ is $\sigma(\mathcal{M})$-reducing. Define $(\sigma_0, T_0):=(\sigma ,T)|_{\mathcal{D}_{*,T}},$ for simplicity we write $(\sigma, T)$ instead of $(\sigma_0, T_0)$. The Poisson kernel of $T$ is the operator $\Pi: \mathcal{H} \longrightarrow \mathcal{F}(E)\otimes_{\sigma}\mathcal{D}_{*,T}$ defined by $$\Pi h:=\sum_{n\geq0}(I_{E^{\otimes n}}\otimes\Delta_*(T))\widetilde{T}_n^* h,\:\:\: h \in \mathcal{H}.$$

Now we recall an isometric dilation of completely contractive covariant representations of a $C^*$-correspondence due to Muhly and Solel \cite{MS98}. 	
\begin{theorem}\label{dilation}
	Let $(\sigma, T)$ be a completely contractive covariant representation of $E$ on $\mathcal{H}.$ Then the poisson kernel $\Pi$ is a contraction and
	\begin{align*}
		\Pi \sigma(b)=\rho(b)\Pi,\:\: b \in \mathcal{M}\quad  \mbox{and} \:\: (I_{E} \ot \Pi)\widetilde{T}^{*}=\widetilde{S}^{*}\Pi. 
	\end{align*}
	We say that the induced representation $(\rho,S)$ of $E$ induced by $\sigma|_{\mathcal{D}_{*,T}}$ is an {\rm isometric dilation} of $(\sigma, T).$ Moreover, $\Pi$ is an isometry if and only if $(\sigma, T) $ is pure (that is, SOT-$\lim_{n\to\infty} \wT_n\wT_n^*= 0$).
\end{theorem}

\subsection{Completely contractive representations over product system}

The central tool we require a
product system of $C^*$-correspondences (see \cite{F02,DTV2,SZ08, S08}): The {\it product system}
$\be$ is defined by a family of $C^*$-correspondences $\{E_1, \ldots,E_k\},k \in \mathbb N$ and by the unitary
isomorphisms $t_{i,j}: E_i \ot E_j \to E_j \ot E_i$ ($i>j$). Using these identifications, for all
${\bf n}=(n_1, \cdots, n_k) \in \Nk$ the correspondence $\be ({\bf n})$ is identified with $E_1^{\ot^{ n_1}} \ot \cdots \ot E_k^{\ot^{n_k}}.$ We use notations $t_{i,i} = \id_{E_i \ot E_i}$ and $t_{i,j} = t_{j,i}^{-1}$ if $i<j.$


\begin{definition}
	Assume $\be$ to be a product system over $\mathbb{N}_{0}^{k}$.
	\begin{itemize}
		\item[(1)] Let $\sigma$ be a
		representation of $\mathcal M$ on $\mathcal H,$ and let $T^{(i)}:E_i \to B(\mathcal H)$ be a  linear maps for $1 \leq i \leq k.$ The tuple $(\sigma, T^{(1)}, \dots, T^{(k)})$ is  called a {\rm completely bounded (respectively, completely contractive) covariant representation} of a product system $\mathbb{E}$ on $\mathcal{H}$ if each pair $(\sigma, T^{(i)})$ is a {completely bounded (respectively, completely contractive) covariant representation} of $E_i$ on $\mathcal{H}$ and satisfy the commutation relation
		\begin{equation*} \label{rep} \wT^{(i)} (I_{E_i} \ot \wT^{(j)}) = \wT^{(j)} (I_{E_j} \ot \wT^{(i)}) (t_{i,j} \ot I_{\mathcal H}) \quad \quad \mbox{for}\quad 1\leq i,j\leq k.
		\end{equation*}
		\item[(2)] 	The covariant representation $(\sigma, T^{(1)}, \ldots, T^{(k)})$ is called {\rm isometric (respectively, co-isometric)} if each $(\sigma, T^{(i)})$ is an isometric (respectively, co-isometric) representation of $E_i.$ The representation $(\sigma, T^{(1)}, \ldots, T^{(k)})$ is called {\rm pure} if each $(\sigma, T^{(i)})$ is pure representation of $E_i$ on $\mathcal{H}.$
	\end{itemize}
\end{definition}
\noindent For each $1\leq i \leq k$
and $l \in \bn,$ define $\wT^{(i)}_l: E_i^{\ot l}\otimes \mathcal H\to \mathcal H$ by
\[ \wT^{(i)}_l (\xi_1 \ot \cdots \ot \xi_l \ot h) := T^{(i)} (\xi_1) \cdots T^{(i)}(\xi_l) h  \] for all $\xi_1, \ldots, \xi_l \in E_{i}, h \in \mathcal H.$
It is easy to verify that
\begin{equation*}
	\wT^{(i)}_l=\wT^{(i)}(I_{E_i} \otimes \wT^{(i)}) \cdots (I_{E^{\otimes l-1}_i} \otimes  \wT^{(i)}).
\end{equation*}
For $\mathbf{n}=(n_1, \cdots, n_k) \in \mathbb{N}_0^k $, we use notation  $\wT_{\mathbf{n}},$ where $\wT_{\mathbf{n}}:\mathbb{E}(\mathbf{n})\otimes_{\sigma} \mathcal{H} \to \mathcal{H}$ is given by
$$\wT_{\mathbf{n}}=\wT_{n_1}^{(1)}\left(I_{E_1^{\otimes n_1}} \otimes\wT_{n_2}^{(2)}\right) \cdots \left(I_{E_1^{\otimes n_1} \otimes \cdots \otimes E_{k-1}^{\otimes {n_{k-1}}}} \otimes\wT_{n_k}^{(k)}\right).$$
Using Lemma \ref{DSAZ1}, let us define the bi-module map $T_{\mathbf{n}}: \mathbb{E}(\mathbf{n}) \to B(\mathcal{H})$  by $$T_{\mathbf{n}}(\xi)h=\wT_{\mathbf{n}}(\xi \otimes h) \quad \mbox{for}\quad ~ \xi \in \mathbb{E}(\mathbf{n}), h \in \mathcal{H}.$$
Define the {\it Fock module} of $\mathbb{E}$ over $\mathbb{N}^k_0$ by
$\mathcal{F}(\mathbb{E}):=\bigoplus_{\mathbf{n} \in \mathbb{N}^k_0}\mathbb{E}(\mathbf{n}).$ Note that $\mathcal{F}(\mathbb{E})$ is a $C^*$-correspondence with the left action $\phi_{\infty}$ of $\mathcal{M}$ on $\mathcal{F}(\mathbb{E})$ given by $$\phi_{\infty}(b)(\bigoplus_{{\bf{n}} \in \mathbb{N}_0^{k}} \xi_{\bf{n}}) = \bigoplus_{{\bf{n}} \in \mathbb{N}_0^{k}} b\xi_{\bf{n}}, \: b\in \mathcal{M},\xi_{\bf{n}} \in \mathbb{E}(\bf{n}).$$ 
Let $\pi$ be a representation of $\mathcal{M}$ on a Hilbert space $\mathcal{K}.$ For $ i\in \{1, \cdots,k\},$ define an isometric covariant representation $(\rho, S^{(i)})$ of $E_i$ on the Hilbert space $\mathcal{F}(\mathbb{E})\otimes_{\pi} \mathcal{K}$ (cf. \cite{SZ08}) by  $$\rho(a):=\phi_{\infty}(a) \otimes I_{\mathcal{K}}, \:\:a \in \mathcal{M} \quad \mbox{and}\quad S^{(i)}(\xi_i)=T_{{\xi}_i} \otimes I_{\mathcal{K}},\:  \xi_i \in E_i,$$ where $T_{\xi_i}$  denotes the creation operator on $\mathcal{F}(\mathbb{E})$ determined by $\xi_i.$ We say that
$(\rho,S^{(1)}, \dots ,S^{(k)})$ is an induced representation of $\be$ over $\mathbb{N}^k_0$ induced by $\pi.$

Suppose $u=\{u_1,u_2,\dots, u_r\}\subseteq \{1,2,\dots,k\}$ such that $u_1< u_2< \dots < u_r.$ Define $\mathbf{e}(u):=\mathbf{e}_{u_1}+\mathbf{e}_{u_2}+\dots+\mathbf{e}_{u_r},$  where $\mathbf{e}_{i}:=(0,0,\dots,1,0,\dots 0),$ i.e., $1$ at $i^{th}$ place and $0$ otherwise. Next we define the following Brehmer-Solel condition:
\begin{definition}
	Let $(\sigma, T^{(1)}, T^{(2)},\dots, T^{(k)})$ be  a completely contractive representation of $\mathbb{E}$ over $\mathbb{N}^k_0$ on $\mathcal{H}.$ We say that it
	satisfies {\rm Brehmer-Solel condition} (cf. \cite{S08}) if it satisfies the condition 
	\begin{align*}
		\sum_{u \subseteq \{1,2,\dots,k\}}(-1)^{|u|}\widetilde{V}_{\mathbf{e}(u)}\widetilde{V}_{\mathbf{e}(u)}^{*}\geq 0.
	\end{align*}
\end{definition}

The following theorem gave an isometric dilation of completely contractive covariant representations of product system due to Rohilla, Trivedi and Veerabathiran \cite[Theorem 3.5]{RTV}: 

\begin{theorem}\label{6}
	Let $(\sigma, T^{(1)}, T^{(2)},\dots, T^{(k)})$  be a pure completely contractive covariant representation of $\mathbb{E}$ on $\mathcal{H},$ satisfying Brehmer-Solel condition. Then there exist a Hilbert space $\mathcal{D}_{*,k}( T)$ and an isometry $\Pi: \mathcal{H} \to \mathcal{F}(\mathbb{E})\otimes_{\sigma}\mathcal{D}_{*,k}( T)$ such that \begin{align*}
		\Pi\sigma(a)=\rho(a)\Pi \:\:\:\:\mbox{and}\:\:\:\: (I_{E_{i}} \ot \Pi)\widetilde{T}^{(i)*}=\widetilde{S}^{(i)*}\Pi \quad \mbox{for}\quad a \in \mathcal{M}, 1\le i\le k. 
	\end{align*}
	We say that the induced representation $(\rho,S^{(1)},\dots,S^{(k)})$ of $\mathbb{E}$ induced by $\sigma|_{\mathcal{D}_{*,T}}$ is an {\rm isometric dilation} of $(\sigma, T^{(1)},\dots, T^{(k)}).$
\end{theorem}

\section{Ando type dilation for completely contractive
	representations of product system over $\mathbb N^2_0$}

This section aims to construct an explicit isometric dilation of completely contractive covariant representations of product system over $\mathbb{N}_0^2.$ In particular, the isometric dilation of completely contractive covariant representation with finite defect indices, as this will play a key role in the subsequent analysis.

Assume $\be$ to be a product system over $\mathbb N^2_0$. Let $(\sigma, T^{(1)}, T^{(2)})$ be a completely contractive covariant representation  of $\be$ on $\mathcal H.$ Define $\widetilde T:=\wT^{(1)} (I_{E_1} \ot \wT^{(2)}),$ then by Lemma \ref{DSAZ1} $(\sigma, T)$ is also a completely contractive covariant represenation of $E_1 \ot E_2$ on $\mathcal{ H}.$ Note that
\begin{align*}
	&I_{\mathcal H}- \wT^{(1)}
	\wT^{(1)^*}+\wT^{(1)} (I_{E_1} \ot (I_{\mathcal H}-\wT^{(2)}\wT^{(2)^*}))\wT^{(1)^*}\\&=
	I_{\mathcal H}-\wT^{(1)} (I_{E_1} \ot \wT^{(2)})(I_{E_1} \ot \wT^{(2)^*})\wT^{(1)^*}=	I_{\mathcal H}-{\widetilde T}{\widetilde T^*}\\&=	I_{\mathcal H}- \wT^{(2)}
	\wT^{(2)^*}+\wT^{(2)} (I_{E_2} \ot (I_{\mathcal H}-\wT^{(1)}\wT^{(1)^*}))\wT^{(2)^*}.
\end{align*}
This implies that, for $h\in \mathcal{ H}$
\begin{align*}
	&\|\Delta_*(T^{(1)})h\|^2+\|(I_{E_1} \ot \Delta_*(T^{(2)}))\wT^{(1)^*}h\|^2
	\\&=\|(I_{E_2} \ot \Delta_*(T^{(1)}))\wT^{(2)^*}h\|^2+\|\Delta_*(T^{(2)})h\|^2.
\end{align*}
Therefore we get an isometry \begin{align*}
	U:\{\Delta_*(T^{(1)})h\oplus (I_{E_1} \ot \Delta_*(T^{(2)}))\wT^{(1)^*}h:~h\in \mathcal H\}\to\\ \{(I_{E_2} \ot \Delta_*(T^{(1)}))\wT^{(2)^*}h\oplus \Delta_*(T^{(2)})h:~h\in \mathcal H\}
\end{align*}
defined by
\begin{align}\label{equation1} 
	\nonumber	&U(\Delta_*(T^{(1)})h,  (I_{E_1} \ot \Delta_*(T^{(2)}))\wT^{(1)^*}h)\\&=((I_{E_2} \ot \Delta_*(T^{(1)}))\wT^{(2)^*}h, \Delta_*(T^{(2)})h),
\end{align} $h\in \mathcal H.$ Let  $\phi_1$ and $\phi_2$ be the left actions of $\mathcal{M}$ on $E_1$ and $E_2,$ respectively. It is easy to check that for $a \in \mathcal{M}$ and $ h \in \mathcal{H},$ $U$ is a module map, that is,
\begin{align*}
	U\rho^{'}(a)=\rho^{''}(a)U,
\end{align*} where $\rho^{'}(a)=\big(\sigma(a)|_{\mathcal{D}_{*, T^{(1)}}}, \phi_{1}(a)\ot I_{\mathcal{D}_{*, T^{(2)}}}\big)$ and $\rho^{''}(a)=\big( \phi_2(a)\ot I_{\mathcal{D}_{*, T^{(1)}}}, \sigma(a)|_{\mathcal{D}_{*, T^{(2)}}} \big)$ are the left actions of $\mathcal{M}$ on the Hilbert spaces $\mathcal{D}_{*, T^{(1)}} \oplus  (E_1 \ot\mathcal{D}_{*, T^{(2)}})$ and   $(E_2 \ot \mathcal{D}_{*, T^{(1)}}) \oplus  \mathcal{D}_{*, T^{(2)}},$ respectively. Indeed, for $h\in \mathcal{H}$ we have 

\begin{align*}
	&U\rho^{'}(a)(\Delta_*(T^{(1)})h,  (I_{E_1} \ot \Delta_*(T^{(2)}))\wT^{(1)^*}h)\\&=U(\sigma(a)\Delta_*(T^{(1)})h, (\phi_{1}(a)\ot I_{\mathcal{D}_{*, T^{(2)}}})(I_{E_1} \ot \Delta_*(T^{(2)}))\wT^{(1)^*}h)\\&=U(\Delta_*(T^{(1)})\sigma(a)h, (I_{E_1} \ot \Delta_*(T^{(2)}))(\phi_{1}(a)\ot I_{\mathcal{H}})\wT^{(1)^*}h)\\&=U(\Delta_*(T^{(1)})\sigma(a)h, (I_{E_1} \ot \Delta_*(T^{(2)}))\wT^{(1)^*}\sigma(a)h)\\&=((I_{E_2} \ot \Delta_*(T^{(1)}))\wT^{(2)^*}\sigma(a)h, \Delta_*(T^{(2)})\sigma(a)h)\\&=((\phi_{2}(a)\ot I_{\mathcal{D}_{*, T^{(1)}}})(I_{E_2} \ot \Delta_*(T^{(1)}))\wT^{(2)^*}h, \sigma(a)\Delta_*(T^{(2)})h)\\&=\rho^{''}(a)U(\Delta_*(T^{(1)})h,  (I_{E_1} \ot \Delta_*(T^{(2)}))\wT^{(1)^*}h)
\end{align*}
Moreover, if  $ \dim (\mathcal{D}_{*, T^{(1)}} \oplus  (E_1 \ot\mathcal{D}_{*, T^{(2)}}))$ =  $ \dim ((E_2 \ot \mathcal{D}_{*, T^{(1)}}) \oplus  \mathcal{D}_{*, T^{(2)}} )$ which is finite, it follows that $U$ extends to a
unitary module map, denoted again by $U,$ from $\mathcal{D}_{*, T^{(1)}} \oplus  (E_1 \ot\mathcal{D}_{*, T^{(2)}})$ to $(E_2 \ot \mathcal{D}_{*, T^{(1)}}) \oplus  \mathcal{D}_{*, T^{(2)}}. $ In
particular, there exists a unitary 
\begin{equation*}
	U: = \begin{bmatrix}A&B\\C&\widetilde{D}\end{bmatrix}:
	\mathcal{D}_{*, T^{(1)}} \oplus  (E_1 \ot\mathcal{D}_{*, T^{(2)}})\to (E_2 \ot \mathcal{D}_{*, T^{(1)}}) \oplus  \mathcal{D}_{*, T^{(2)}},
\end{equation*}
and it satisfies Equation (\ref{equation1}), where $ A:\mathcal{D}_{*, T^{(1)}} \to E_2 \ot \mathcal{D}_{*, T^{(1)}},$ $ B:E_1 \ot \mathcal{D}_{*, T^{(2)}} \to E_2 \ot \mathcal{D}_{*, T^{(1)}},$ $ C:\mathcal{D}_{*, T^{(1)}} \to \mathcal{D}_{*, T^{(2)}}$ and $ \widetilde{D} : E_1 \ot\mathcal{D}_{*, T^{(2)}}\to  \mathcal{D}_{*, T^{(2)}}.$ 
For $\eta \in \mathcal{D}_{*, T^{(1)}} \oplus  (E_1 \ot\mathcal{D}_{*, T^{(2)}})$ we have
\begin{equation}\label{inner}
	U\rho^{'}(a)\eta =\rho^{''}(a)U\eta.
\end{equation}
Therefore $U$ is a module map, and hence each entry of $U$ is also a module map. That is, $A,B,C$ and $\widetilde{D}$ are module maps. In particular, $A=P_{E_2 \ot \mathcal{D}_{*, T^{(1)}}}U|_{\mathcal{D}_{*, T^{(1)}}}:\mathcal{D}_{*, T^{(1)}} \to E_2 \ot \mathcal{D}_{*, T^{(1)}} ,$  we have $A\sigma(a)|_{\mathcal{D}_{*, T^{(1)}}}=(\phi_2(a)\ot I_{\mathcal{D}_{*, T^{(1)}}})A$ for $a\in \mathcal{M}.$ That is,\begin{equation*}
	\sigma(a)|_{\mathcal{D}_{*, T^{(1)}}}A^{*}=A^{*}(\phi_2(a)\ot I_{\mathcal{D}_{*, T^{(1)}}}),
\end{equation*} where $P_{E_2 \ot \mathcal{D}_{*, T^{(1)}}}$ is the orthogonal projection of $E_2\ot \mathcal{ H}$ onto $E_2 \ot \mathcal{D}_{*, T^{(1)}}.$ Indeed, by using Equation (\ref{inner}) for $h\in \mathcal{D}_{*, T^{(1)}}$ we obtain \begin{align*}\label{A}
	A\sigma(a)h=P_{E_2 \ot \mathcal{D}_{*, T^{(1)}}}U\sigma(a)h&=P_{E_2 \ot \mathcal{D}_{*, T^{(1)}}}(\phi_2(a)\ot I_{\mathcal{D}_{*, T^{(1)}}})Uh\\&=(\phi_2(a)\ot I_{\mathcal{D}_{*, T^{(1)}}})Ah.
\end{align*}

From Lemma \ref{DSAZ1}, we define a completely contractive covariant representation $(\psi_1, X)$ of $E_2$ on the Hilbert space $\mathcal{D}_{*, T^{(1)}}$ by $$X(\xi_2)h:=A^*(\xi_2\ot h),$$ where $\xi_2\in E_2, h\in \mathcal{D}_{*, T^{(1)}}$ and $\psi_1(a):=\sigma(a)|_{\mathcal{D}_{*, T^{(1)}}}, a\in \mathcal{M}.$ Similarly, $(\sigma_1,D)$ is a completely contractive covariant representation of $E_1$ on the Hilbert space $\mathcal{D}_{*, T^{(2)}},$ where $\sigma_1(a):=\sigma(a)|_{\mathcal{D}_{*, T^{(2)}}}, a\in \mathcal{M}.$

The following lemma is crucial to our discussion.

\begin{lemma}\label{identity}
	Assume $\be$ to be a product system over $\mathbb N^2_0$. Let $(\sigma, T^{(1)}, T^{(2)})$ be a completely contractive covariant representation  of $\be$ on $\mathcal H$ such that $(\sigma,T^{(1)})$ is pure and  $ \dim (\mathcal{D}_{*, T^{(1)}} \oplus  (E_1 \ot\mathcal{D}_{*, T^{(2)}}))$ =  $ \dim ((E_2 \ot \mathcal{D}_{*, T^{(1)}}) \oplus  \mathcal{D}_{*, T^{(2)}} )<\infty.$  Then 
	\begin{align*}&
		(I_{E_2} \ot \Delta_*(T^{(1)}))\wT^{(2)^*}= A\Delta_*(T^{(1)})
		\\&\hspace{0.2in}
		+  \sum_{n= 0}^\infty
		B(I_{E_1}\ot \widetilde{D}_n)(I_{E_1^{\otimes n+1}}\ot C\Delta_*(T^{(1)}))\wT^{(1)^*}_{n+1},
	\end{align*}
	where $U=\begin{bmatrix}A&B\\C&\widetilde{D}\end{bmatrix}$ defined as above, $\widetilde{D}_n=\widetilde{D}(I_{E_1} \otimes \widetilde{D}) \dots (I_{E_1^{\ot(n-1)}} \otimes  \widetilde{D})$ and the series converges with respect to the strong operator topology.
\end{lemma}
\begin{proof}
	Let $h \in \mathcal H$ we get
	\[\begin{bmatrix}A&B\\C&\widetilde{D}\end{bmatrix} \begin{bmatrix} \Delta_*(T^{(1)})h\\ (I_{E_1} \ot \Delta_*(T^{(2)}))\wT^{(1)^*}h \end{bmatrix}= \begin{bmatrix} (I_{E_2} \ot \Delta_*(T^{(1)}))\wT^{(2)^*}h \\  \Delta_*(T^{(2)})h\end{bmatrix}.\] It follows that
	\begin{equation} \label{one} (I_{E_2} \ot \Delta_*(T^{(1)}))\wT^{(2)^*}h= A\Delta_*(T^{(1)})h+
		B(I_{E_1} \ot \Delta_*(T^{(2)}))\wT^{(1)^*}h
	\end{equation}and
	\begin{equation*}
		\Delta_*(T^{(2)})h= C\Delta_*(T^{(1)})h+
		\widetilde{D}(I_{E_1} \ot \Delta_*(T^{(2)}))\wT^{(1)^*}h.
	\end{equation*} This implies that
	\begin{align*}
		&(I_{E_1}\ot \Delta_*(T^{(2)}))\wT^{(1)^*}h= (I_{E_1}\ot C\Delta_*(T^{(1)}))\wT^{(1)^*}h+\\& \quad \quad
		(I_{E_1}\ot \widetilde{D}(I_{E_1} \ot \Delta_*(T^{(2)}))\wT^{(1)^*})\wT^{(1)^*}h,	
	\end{align*}
	and further we get\begin{align*}
		&(I_{E_1^{\otimes 2}}\ot \Delta_*(T^{(2)}))\wT^{(1)^*}_2h\\&= (I_{E_1^{\otimes 2}}\ot C\Delta_*(T^{(1)}))\wT^{(1)^*}_2h+
		(I_{E_1^{\otimes 2}}\ot \widetilde{D}(I_{E_1} \ot \Delta_*(T^{(2)}))\wT^{(1)^*})\wT^{(1)^*}_2h,
	\end{align*} and hence
	\begin{align*}
		&(I_{E_1^{\otimes 2}}\ot \Delta_*(T^{(2)}))\wT^{(1)^*}_2h\\&= (I_{E_1^{\otimes 2}}\ot C\Delta_*(T^{(1)}))\wT^{(1)^*}_2h+
		(I_{E_1^{\otimes 2}}\ot \widetilde{D}(I_{E_1} \ot \Delta_*(T^{(2)})))\wT^{(1)^*}_3h.
	\end{align*}From Equation ~\eqref{one} we obtain
	\[\begin{split}&
		(I_{E_2} \ot \Delta_*(T^{(1)}))\wT^{(2)^*}h\\&= A\Delta_*(T^{(1)})h+
		B(I_{E_1} \ot \Delta_*(T^{(2)}))\wT^{(1)^*}h \\ & =  A\Delta_*(T^{(1)})h+
		B(I_{E_1}\ot C\Delta_*(T^{(1)}))\wT^{(1)^*}h   \\& \quad \quad +
		B(I_{E_1}\ot \widetilde{D}(I_{E_1} \ot \Delta_*(T^{(2)}))\wT^{(1)^*})\wT^{(1)^*}h
		\\ & =  A\Delta_*(T^{(1)})h+
		B(I_{E_1}\ot C\Delta_*(T^{(1)}))\wT^{(1)^*}h  \\&  \quad \quad +
		B(I_{E_1}\ot \widetilde{D})(I_{E_1^{\otimes 2}} \ot \Delta_*(T^{(2)}))\wT^{(1)^*}_2h
		\\&=A\Delta_*(T^{(1)})h+
		B(I_{E_1}\ot C\Delta_*(T^{(1)}))\wT^{(1)^*}h  \\&  \quad \quad +
		B(I_{E_1}\ot \widetilde{D})(I_{E_1^{\otimes 2}}\ot C\Delta_*(T^{(1)}))\wT^{(1)^*}_2h \\& \quad \quad +
		B(I_{E_1}\ot \widetilde{D})
		(I_{E_1^{\otimes 2}}\ot \widetilde{D}(I_{E_1} \ot \Delta_*(T^{(2)})))\wT^{(1)^*}_3h 
		\\&=A\Delta_*(T^{(1)})h+
		B(I_{E_1}\ot C\Delta_*(T^{(1)}))\wT^{(1)^*}h  \\& \quad \quad +
		B(I_{E_1}\ot \widetilde{D})(I_{E_1^{\otimes 2}}\ot C\Delta_*(T^{(1)}))\wT^{(1)^*}_2h  \\& \quad \quad +
		B(I_{E_1}\ot \widetilde{D})
		(I_{E_1^{\otimes 2}}\ot \widetilde{D})(I_{E_1^{\otimes 3}}\ot\Delta_*(T^{(2)}))\wT^{(1)^*}_3h\\&=A\Delta_*(T^{(1)})h+
		B(I_{E_1}\ot C\Delta_*(T^{(1)}))\wT^{(1)^*}h  \\& \quad \quad +
		B(I_{E_1}\ot \widetilde{D})(I_{E_1^{\otimes 2}}\ot C\Delta_*(T^{(1)}))\wT^{(1)^*}_2h  \\& \quad \quad +
		B(I_{E_1}\ot \widetilde{D}(I_{E_1}\ot \widetilde{D}))
		(I_{E_1^{\otimes 3}}\ot\Delta_*(T^{(2)}))\wT^{(1)^*}_3h\\&=A\Delta_*(T^{(1)})h+
		B(I_{E_1}\ot C\Delta_*(T^{(1)}))\wT^{(1)^*}h  \\& \quad \quad +
		B(I_{E_1}\ot \widetilde{D})(I_{E_1^{\otimes 2}}\ot C\Delta_*(T^{(1)}))\wT^{(1)^*}_2h \\& \quad \quad +
		B(I_{E_1}\ot \widetilde{D}_{2})
		(I_{E_1^{\otimes 3}}\ot\Delta_*(T^{(2)}))\wT^{(1)^*}_3h.
	\end{split}\] 
	Therefore iteratively for each $m\in\mathbb N$, we have
	\begin{align*}
		(I_{E_2} \ot \Delta_*(T^{(1)}))\wT^{(2)^*}h= &  A\Delta_*(T^{(1)})h
		\\& +\sum_{n= 0}^m 
		B(I_{E_1}\ot \widetilde{D}_{n})(I_{E_1^{\otimes n+1}}\ot C\Delta_*(T^{(1)}))\wT^{(1)^*}_{n+1}h  \\&+ 
		B(I_{E_1}\ot \widetilde{D}_{m+1})(I_{E_1^{\otimes m+2}}\ot \Delta_*(T^{(1)}))\wT^{(1)^*}_{m+2}h.
	\end{align*}
	Since $(\sigma,T^{(1)})$ is pure and $\|\widetilde{D}\| \leq 1,$ it gives
	\begin{align*}
		\|(I_{E_2} \ot & \Delta_*(T^{(1)}))\wT^{(2)^*}h - A\Delta_*(T^{(1)})h-\\& \sum_{n= 0}^m 
		B(I_{E_1}\ot \widetilde{D}_{n})(I_{E_1^{\otimes n+1}}\ot C\Delta_*(T^{(1)}))\wT^{(1)^*}_{n+1}h\|\\& =\|B(I_{E_1}\ot \widetilde{D}_{m+1})(I_{E_1^{\otimes m+2}}\ot \Delta_*(T^{(1)}))\wT^{(1)^*}_{m+2}h\| \\&
		\leq\|\wT^{(1)^*}_{m+2}h\|  \to 0~\mbox{as}~m\to\infty.
	\end{align*}  This complets the proof.
\end{proof}

\begin{observation}
	Under the above assumptions, the conclusion of Lemma \ref{identity} remains valid even when we incorporate the possibility
	$ \dim (\mathcal{D}_{*, T^{(1)}} \oplus  (E_1 \ot\mathcal{D}_{*, T^{(2)}}))$ or $ \dim ((E_2 \ot \mathcal{D}_{*, T^{(1)}}) \oplus  \mathcal{D}_{*, T^{(2)}} )$ is an infinite dimensional space. The conclusion of Lemma \ref{identity} remains valid. If  $ \dim (\mathcal{D}_{*, T^{(1)}} \oplus  (E_1 \ot\mathcal{D}_{*, T^{(2)}}))=\infty$ or $ \dim ((E_2 \ot \mathcal{D}_{*, T^{(1)}}) \oplus  \mathcal{D}_{*, T^{(2)}} )=\infty,$ 
	then there exists an infinite dimensional 
	Hilbert space $\mathcal{D}$ such that the isometry 
	\begin{align*}&
		U:\{\Delta_*(T^{(1)})h\oplus (I_{E_1} \ot \Delta_*(T^{(2)}))\wT^{(1)^*}h:~h\in \mathcal H\} \oplus \{0_\mathcal{D}\}\to \\& \:\:\:\:\:\:\:\: \{(I_{E_2} \ot \Delta_*(T^{(1)}))\wT^{(2)^*}h\oplus \Delta_*(T^{(2)})h:~h\in \mathcal H\} \oplus \{0_\mathcal{D}\}
	\end{align*} is defined by 
	\begin{align*} 
		&U(\Delta_*(T^{(1)})h,  (I_{E_1} \ot \Delta_*(T^{(2)}))\wT^{(1)^*}h,0_\mathcal{D})\\&=((I_{E_2} \ot \Delta_*(T^{(1)}))\wT^{(2)^*}h, \Delta_*(T^{(2)})h,0_\mathcal{D}),
	\end{align*} for all $h\in \mathcal H$ 
	extends to a unitary, denoted again by $U,$ from
	$\mathcal{D}_{*, T^{(1)}} \oplus  (E_1 \ot\mathcal{D}_{*, T^{(2)}}) \oplus \mathcal{D}$  to
	$  (E_2 \ot\mathcal{D}_{*, T^{(1)}}) \oplus \mathcal{D}_{*, T^{(2)}} \oplus \mathcal{D}.$ Now we proceed similarly
	with the unitary matrix \begin{equation*}
		U = \begin{bmatrix}A&B\\C&\widetilde{D}\end{bmatrix}:
		\mathcal{D}_{*, T^{(1)}} \oplus  (E_1 \ot\mathcal{D}_{*, T^{(2)}}) \oplus \mathcal{D}\to (E_2 \ot \mathcal{D}_{*, T^{(1)}}) \oplus  \mathcal{D}_{*, T^{(2)}}\oplus \mathcal{D} ,
	\end{equation*} defined by 	\begin{align*} &U(\Delta_*(T^{(1)})h,  (I_{E_1} \ot \Delta_*(T^{(2)}))\wT^{(1)^*}h,0_\mathcal{D})\\&=((I_{E_2} \ot \Delta_*(T^{(1)}))\wT^{(2)^*}h, \Delta_*(T^{(2)})h,0_\mathcal{D}),\end{align*} to obtain the same conclusion as in Lemma \ref{identity}.
\end{observation}

\subsection{Transfer Function}

Suppose that $\mathbb{E}$ is a product system over $\mathbb{N}_0^2,$ and $\pi $ is a representation of $\mathcal{M}$ on $\mathcal{K}.$ Let $\Theta: E_2 \longrightarrow B(\mathcal{K}, \mathcal{F}(E_1) \otimes_{\pi} \mathcal{K})$ be a completely bounded { bi-module map}. Define a bounded linear map $\widetilde{\Theta}: E_2 \otimes \mathcal{K} \longrightarrow \mathcal{F}(E_1) \otimes_{\pi} \mathcal{K}$  by $\widetilde{\Theta}(\xi \otimes h)=\Theta(\xi)h$ for all $\xi\in E_2,h\in \mathcal{K},$ and it satisfies $\widetilde{\Theta}(\phi_2(a) \ot I_{ \mathcal{K}})=\rho(a)\widetilde{\Theta},a\in \mathcal{M},$ where $\phi_2$ is the left action of $\mathcal{M}$ on $E_2.$
We define a corresponding completely bounded bi-module map \cite{DTV1} $M_{\Theta} :E_2  \longrightarrow B(\mathcal{F}(E_1) \otimes_{\pi} \mathcal{K})$ by
\begin{align*}
	M_{\Theta}(\xi) (S_n^{(1)}(\xi_n) h)=\widetilde{S}_n^{(1)}(I_{E_1^{\otimes n}} \otimes \widetilde{\Theta})(t_{2,1}^{(1,n)} \otimes I_{\mathcal{K}})(\xi \otimes \xi_n \otimes h),
\end{align*}
where  $\xi_n \in E_1^{\otimes n}, \xi \in E_2, h\in \mathcal{K}, n \in \mathbb{N}_0$ and $t_{2,1}^{(1,n)}: E_2 \ot E_1^{\ot n}  \rightarrow E_1^{\ot n} \ot E_2 $ is an isomorphism which is a composition of the isomorphisms $\{t_{i,j}\: :\: 1 \leq i, j \leq 2\}.$
Clearly  $M_{\Theta}(\xi)|_{\mathcal{K}}=\Theta(\xi)$ for each $\xi \in E_2,$  $(\rho,{M}_{\Theta})$ is a c.b.c. representation of $E_2$ on $\mathcal{F}(E_1) \otimes_{\pi} \mathcal{K},$ and it satisfies
\begin{align*}
	M_{\Theta}(\xi)\left(\bigoplus_{n \in \mathbb{N}_0}\xi_n \otimes h_n\right)=\sum_{n \in \mathbb{N}_0}\widetilde{S}_n^{(1)}(I_{E_1^{\otimes n}} \otimes \widetilde{\Theta})(t_{2,1}^{(1,n)} \otimes I_{\mathcal{K}})(\xi \otimes \xi_n \otimes h_n),
\end{align*} 
where $\xi \in E_2, \xi_n \in E_1^{\otimes n}, h_n \in \mathcal{K}.$  Let  $(\rho, S)$ be the induced representation of $E_1$ induced by $\pi,$ then $(\rho, S, M_{\Theta})$ is a c.b.c. representation of $\mathbb{E}$ on $\mathcal{F}(E_1) \otimes _{\pi} \mathcal{K} .$
Note that  $(\rho, M_{\Theta})$ is an isometric covariant representation if and only if $\widetilde{\Theta}$  is an isometry.

With the help of unitary module map $U$ as we defined above, we define a bounded linear map $\widetilde{\Theta}: E_{2}\ot \mathcal{D}_{*, T^{(1)}} \to F(E_{1}) \ot \mathcal{D}_{*, T^{(1)}}  $ by \begin{equation*}
	\widetilde{\Theta}:=A^{*}\oplus \bigoplus_{n=0}^{\infty}(I_{E_{1}^{\ot n+1}}\ot  C^{*})(I_{E_1} \ot \widetilde{D}^*_n)B^{*}.
\end{equation*} First we will show that it is well defined, it is enough to prove that $\|\bigoplus_{n=0}^{\infty}(I_{E_{1}^{\ot n+1}} \ot  C^{*})(I_{E_1} \ot \widetilde{D}^*_n)B^{*}(\eta)\|^{2}< \infty$ for $\eta \in E_{2}\ot \mathcal{D}_{*, T^{(1)}}.$ Indeed, using $UU^*=I$ we have
\begin{align*}
	&\|\bigoplus_{n=0}^{\infty}(I_{E_{1}^{\ot n+1}}\ot  C^{*})(I_{E_1} \ot \widetilde{D}^*_n)B^{*}(\eta)\|^{2}\\&=\sum_{n= 0}^{\infty}\langle B(I_{E_1} \ot \widetilde{D}_{n})(I_{E_{1}^{\ot n+1}}\ot  CC^{*})(I_{E_1} \ot \widetilde{D}^*_n)B^{*}\eta,\eta \rangle\\&=\sum_{n= 0}^{\infty}\langle B(I_{E_1} \ot \widetilde{D}_{n})(I_{E_{1}^{\ot n+1}}\ot  (I_{\mathcal{D}_{*, T^{(2)}}}-\widetilde{D}\widetilde{D}^{*}))(I_{E_1} \ot \widetilde{D}^*_n)B^{*}\eta,\eta \rangle\\&=\sum_{n= 0}^{\infty}\langle\big(B(I_{E_1} \ot \widetilde{D}_{n}\widetilde{D}^*_n)B^{*}-B(I_{E_1} \ot \widetilde{D}_{n+1}\widetilde{D}_{n+1}^*)B^{*}\big)\eta,\eta \rangle\\&=\langle BB^{*}-\lim_{n \rightarrow \infty}B(I_{E_1} \ot \widetilde{D}_{n}\widetilde{D}^*_n)B^{*}\eta,\eta \rangle\\&=
	\|B^{*}\eta\|^{2}-\lim_{n \rightarrow \infty}\langle B(I_{E_1} \ot \widetilde{D}_{n}\widetilde{D}^*_n)B^{*}\eta,\eta \rangle<\|B^{*}\eta\|^{2}.
\end{align*}
Here the last inequality follows by the fact that $B$ and $\widetilde{D}$ are contractions. This shows that $\widetilde{\Theta}$ is well defined and contractive.

Since $(\sigma_1, D)$ is a completely contractive covariant representation of $E_1$ on $\mathcal{D}_{*, T^{(2)}},$ from Theorem \ref{contra}, $\mathcal{D}_{*, T^{(2)}}=\mathcal{D}^{(1)}_{*, T^{(2)}} \oplus \mathcal{D}^{(2)}_{*, T^{(2)}}$ has unique orthogonal decomposition such that $\mathcal{D}^{(1)}_{*, T^{(2)}},$ $\mathcal{D}^{(2)}_{*, T^{(2)}}$ reduces $(\sigma_1, D),$  $(\pi_1,V^{(1)}):=(\sigma_1, D)|_{\mathcal{D}^{(1)}_{*, T^{(2)}}}$ is an isometric as well as co-isometric covariant representation and  $(\pi_2,V^{(2)}):=(\sigma_1, D)|_{\mathcal{D}^{(2)}_{*, T^{(2)}}}$ is a completely non-unitary covariant representation. That is,  $\widetilde{D}^{*}\widetilde{D}|_{E_1 \ot\mathcal{D}^{(1)}_{*, T^{(2)}}}=I_{E_1 \ot\mathcal{D}^{(1)}_{*, T^{(2)}}}.$   Also using the fact that $U$ is unitary, we have $B^{*}B+\widetilde{D}^{*}\widetilde{D}=I_{E_1 \ot\mathcal{D}_{*, T^{(2)}}},$ therefore $B^{*}B|_{E_1 \ot\mathcal{D}^{(1)}_{*, T^{(2)}}}=0,$ i.e., $E_1 \ot\mathcal{D}^{(1)}_{*, T^{(2)}}\subseteq ker(B),$ or equivalently, ${Im(B^{*})}\subseteq E_1 \ot\mathcal{D}^{(2)}_{*, T^{(2)}}.$ Hence $(I_{E_1} \ot \widetilde{D}_{n}\widetilde{D}_n^*)B^{*}=(I_{E_1} \ot \widetilde{V}_n^{(2)}\widetilde{V}_n^{(2)*})B^*.$  Finally, we can rewrite $\lim_{n \rightarrow \infty} B(I_{E_1} \ot \widetilde{D}_n\widetilde{D}_n^*)B^{*}=\lim_{n \rightarrow \infty} B(I_{E_1} \ot \widetilde{V}_{n}^{(2)}\widetilde{V}_{n}^{(2)*})B^*.$
\begin{observation}
	$\widetilde{\Theta}$ is an isometry if and only if $(\pi_2, V^{(2)})$ is pure.
\end{observation}
Define ${\Theta}: E_{2} \to B(\mathcal{D}_{*, T^{(1)}},F(E_{1}) \ot \mathcal{D}_{*, T^{(1)}})$ by ${\Theta}(\eta)h=\widetilde{\Theta}(\eta\ot h),$ where $\eta \in E_{2}$ and $ h \in\mathcal{D}_{*, T^{(1)}}.$ We need to show that $\Theta$ is a bi-module map. That is, $\Theta(a.\eta. b)=\rho(a)\Theta(\eta)\rho(b).$ Let $ h \in \mathcal{D}_{*, T^{(1)}}$ we have \begin{align*}&
	\widetilde{\Theta}(a.\eta. b\ot h)= A^{*}(a.\eta. b\ot h)\oplus \bigoplus_{n=0}^{\infty}(I_{E_{1}^{\ot n+1}}\ot  C^{*})(I_{E_1} \ot \widetilde{D}^*_{n})B^{*}(a.\eta. b\ot h)\\&=A^{*}(a.\eta \ot b .h)\oplus \bigoplus_{n=0}^{\infty}(I_{E_{1}^{\ot n+1}}\ot  C^{*})(I_{E_1} \ot \widetilde{D}^*_{n})B^{*}(a.\eta \ot b .h)\\&=\big(A^{*}(\phi_2(a) \ot I_{ \mathcal{D}_{*, T^{(1)}}})\oplus \\& \quad \quad \bigoplus_{n=0}^{\infty}(I_{E_1} \ot (I_{E_{1}^{\ot n}}\ot  C^{*})\widetilde{D}^*_{n})B^{*}(\phi_{2}(a)  \ot I_{ \mathcal{D}_{*, T^{(1)}}})\big)(\eta \ot b. h)\\&=\sigma(a)|_{\mathcal{D}_{*, T^{(1)}}}A^{*}(\eta \ot b. h)\oplus \\& \quad \quad \bigoplus_{n=0}^{\infty}(I_{E_{1}^{\ot n+1}}\ot  C^{*})(I_{E_1} \ot \widetilde{D}^*_{n})(\phi_{1}(a) \ot I_{ \mathcal{D}_{*, T^{(2)}}})B^{*}(\eta \ot b. h)\\&=\sigma(a)|_{\mathcal{D}_{*, T^{(1)}}}A^{*}(\eta \ot b. h)\oplus \\& \quad \quad \bigoplus_{n=0}^{\infty}(I_{E_{1}^{\ot n+1}}\ot  C^{*})(\phi_1^{n+1}(a) \ot I_{ \mathcal{D}_{*, T^{(2)}}})(I_{E_1} \ot \widetilde{D}_n^{*})B^{*}(\eta \ot b. h)\\&=\sigma(a)|_{\mathcal{D}_{*, T^{(1)}}}A^{*}(\eta \ot b. h)\oplus \\& \quad \quad \bigoplus_{n=0}^{\infty}(\phi_1^{ n+1}(a) \ot I_{ \mathcal{D}_{*, T^{(1)}}})(I_{E_{1}^{\ot n+1}}\ot  C^{*})(I_{E_1} \ot \widetilde{D}_n^{*})B^{*}(\eta \ot b. h).
\end{align*} Finally, we obtain \begin{align*}
	&\widetilde{\Theta}(a.\eta. b\ot h)=\big(\sigma(a)|_{\mathcal{D}_{*, T^{(1)}}}A^{*}\oplus \\& \quad \quad  \bigoplus_{n=0}^{\infty}(\phi_1^{ n+1}(a) \ot I_{ \mathcal{D}_{*, T^{(1)}}})(I_{E_1} \ot (I_{E_{1}^{\ot n}}\ot  C^{*})\widetilde{D}_n^{*})B^{*}\big)( \eta .b\ot h)\\&=\big(\big(\sigma(a)|_{\mathcal{D}_{*, T^{(1)}}}\oplus \bigoplus_{n=0}^{\infty}(\phi_1^{n+1}(a) \ot I_{ \mathcal{D}_{*, T^{(1)}}})\big)(A^{*}\oplus \\& \quad \quad \bigoplus_{n=0}^{\infty}(I_{E_1} \ot (I_{E_{1}^{\ot n}}\ot  C^{*})\widetilde{D}_n^{*})B^{*}))\big( \eta .b\ot h)\\&=\rho(a)\Theta(\eta)\rho(b)h.
\end{align*} Thus $\Theta$ is a bi-module map. From \cite[Lemma 2.2]{DTV1}, $(\rho,S^{(1)},M_{\Theta})$ is a completely contractive covariant representation of $\mathbb{E}$ over $\mathbb{N}_0^2$ on 	$\mathcal{F}(E_1) \otimes \mathcal{D}_{*, T^{(1)}}.$ The contractive linear map $\widetilde{\Theta}$ is called the {\it transfer function} of unitary module map $U.$

\subsection{Ando Dilations}

We are ready to prove our main result, a generalization of \cite[Theorem 3.1]{BJ17}.   
\begin{theorem}\label{DSAR2}
	Assume $\be$ to be a product system over $\mathbb N^2_0$. Let $(\sigma, T^{(1)}, T^{(2)})$ be a completely contractive covariant representation of $\be$ on $\mathcal H$ such that $(\sigma,T^{(1)})$ is pure and  $ \dim (\mathcal{D}_{*, T^{(1)}} \oplus  (E_1 \ot\mathcal{D}_{*, T^{(2)}}))$ =  $ \dim ((E_2 \ot \mathcal{D}_{*, T^{(1)}}) \oplus  \mathcal{D}_{*, T^{(2)}} )<\infty.$ Then there exist an isometry $\Pi_{1}: \mathcal{H} \longrightarrow \mathcal{F}(E_{1})\otimes \mathcal{D}_{*, T^{(1)}}$ and a completely contractive bi-module map ${\Theta}: E_{2} \to B(\mathcal{D}_{*, T^{(1)}},F(E_{1}) \ot \mathcal{D}_{*, T^{(1)}})$ such that 	
	\begin{align*}
		\Pi_{1} \sigma(a)=\rho(a)\Pi_1, (I_{E_1} \ot \Pi_{1})\widetilde{T}^{(1)*}=\widetilde{S}^{(1)*}\Pi_{1} \:\: \mbox{and} \:\: {(I_{E_{2}} \ot\Pi_{1}) \wT^{(2)}}^{*} = \widetilde{M}^{*}_{\Theta}\Pi_{1},
	\end{align*} for all $a\in \mathcal{M}.$
\end{theorem}
\begin{proof}
	Let $(\sigma, T^{(1)})$ be a pure completely contractive covariant representation of $E_1$ on $\mathcal H,$ then from Theorem \ref{dilation}, there exists an isometry $\Pi_{1}: \mathcal{H} \longrightarrow \mathcal{F}(E_{1})\otimes \mathcal{D}_{*, T^{(1)}}$ defined by $$\Pi_{1} h=\sum_{k\geq0}(I_{E_{1}^{\otimes k}}\otimes\Delta_*(T^{(1)}))\wT^{(1)^*}_{k}h,\:\:\: h \in \mathcal{H}$$ such that
	\begin{align*}
		\Pi_{1} \sigma(a)=\rho(a)\Pi_1 \:\: \mbox{and} \:\: (I_{E_1} \ot \Pi_{1})\widetilde{T}^{(1)*}=\widetilde{S}^{(1)*}\Pi_{1},
	\end{align*}
	where $(\rho,S^{(1)})$ is an induced representation of $E_1$ induced by $\sigma|_{\mathcal{D}_{*,T^{(1)}}}.$ We only need to show that $ {(I_{E_{2}} \ot\Pi_{1}) \wT^{(2)}}^{*} = \widetilde{M}^{*}_{\Theta}\Pi_{1}.$ 
	Define a completely contractive bi-module map ${\Theta}: E_{2} \to B(\mathcal{D}_{*, T^{(1)}},F(E_{1}) \ot \mathcal{D}_{*, T^{(1)}})$ by $\Theta(\xi)h=\widetilde{\Theta}(\xi \ot h)$ where $\xi\in E_2,h\in \mathcal{D}_{*, T^{(1)}}$ and $
	\widetilde{\Theta}=A^{*}\oplus \bigoplus_{m=0}^{\infty}(I_{E_{1}^{\ot m+1}}\ot  C^{*})(I_{E_1} \ot \widetilde{D}_m^{*})B^{*}.$ Define a corresponding contractive linear map $\widetilde{M}_{\Theta} :E_2 \ot F(E_{1}) \ot \mathcal{D}_{*, T^{(1)}}  \longrightarrow F(E_{1}) \ot \mathcal{D}_{*, T^{(1)}} $ by $$\widetilde{M}_{\Theta}=\widetilde{S}^{(1)}_n(I_{E_1^{\otimes n}} \otimes \widetilde{\Theta})(t_{2,1}^{(1,n)} \otimes I_{\mathcal{D}_{*, T^{(1)}}}),$$ and it satisfies $\widetilde{M}_{\Theta}(\phi_2(a) \ot I_{F(E_{1}) \ot \mathcal{D}_{*, T^{(1)}} })=\rho(a)\widetilde{M}_{\Theta}$ for $a \in \mathcal{M}.$ For $ \eta_n \in {E_2}\ot E_1^{\otimes n}\ot \mathcal{D}_{*, T^{(1)}},h\in \mathcal{ H}$ and from Lemma \ref{identity} we have
	
	\begin{align*} &
		\langle\widetilde{M}^{*}_{\Theta}\Pi_{1}h ,\eta_n\rangle \\&=\langle\sum_{k= 0}^\infty (I_{E_1^{\otimes k}}\ot \Delta_*(T^{(1)}))\wT^{(1)^*}_{k}h ,\widetilde{S}_n^{(1)}(I_{E_1^{\otimes n}} \otimes \widetilde{\Theta})(t_{2,1}^{(1,n)} \otimes I_{\mathcal{D}_{*, T^{(1)}} })\eta_n\rangle \\&=\langle(I_{E_1^{\otimes n}}\ot \Delta_*(T^{(1)}))\wT^{(1)^*}_{n}h,(I_{E_1^{\otimes n}}\ot A^{*})(t_{2,1}^{(1,n)} \otimes I_{\mathcal{D}_{*, T^{(1)}} })\eta_n\rangle\\&  \quad \quad +\sum_{m=0}^{\infty}\langle(I_{E_1^{\otimes n+m+1}}\ot \Delta_*(T^{(1)}))\wT^{(1)^*}_{n+m+1}h,\\& \quad \quad\quad (I_{E_1^{\otimes n}}\ot (I_{E_1} \ot (I_{E_{1}^{\ot m}}\ot  C^{*})\widetilde{D}_m^{*})B^{*})(t_{2,1}^{(1,n)} \otimes I_{\mathcal{D}_{*, T^{(1)}} })\eta_n\rangle \\&= \langle(I_{E_1^{\otimes n}}\ot A\Delta_*(T^{(1)}))\wT^{(1)^*}_{n}h,(t_{2,1}^{(1,n)} \otimes I_{\mathcal{D}_{*, T^{(1)}} })\eta_n\rangle\\&  \quad \quad +\sum_{m=0}^{\infty}\langle(I_{E_1^{\otimes n}}\ot( B(I_{E_1}\ot \widetilde{D}_{m})))(I_{E_1^{\otimes n+m+1}}\ot C\Delta_*(T^{(1)}))\wT^{(1)^*}_{n+m+1}h,\\& \quad \quad\quad (t_{2,1}^{(1,n)} \otimes I_{\mathcal{D}_{*, T^{(1)}} })\eta_n\rangle\\&= \langle\big(I_{E_1^{\otimes n}}\ot A\Delta_*(T^{(1)})+\sum_{m=0}^{\infty}(I_{E_1^{\otimes n}}\ot( B(I_{E_1}\ot \widetilde{D}_{m})))\\& \quad \quad (I_{E_1^{\otimes n+m+1}}\ot C\Delta_*(T^{(1)}))\wT^{(1)^*}_{m+1})\big)\wT^{(1)^*}_{n}h,(t_{2,1}^{(1,n)} \otimes I_{\mathcal{D}_{*, T^{(1)}} })\eta_n\rangle\\&= \langle (I_{E_1^{\otimes n}}\ot \big(A\Delta_*(T^{(1)})+\sum_{m=0}^{\infty} B(I_{E_1}\ot \widetilde{D}_{m}) (I_{E_1^{\otimes m +1}}\ot C\Delta_*(T^{(1)}))\wT^{(1)^*}_{m+1})\big)\wT^{(1)^*}_{n}h,\\&\quad \quad (t_{2,1}^{(1,n)} \otimes I_{\mathcal{D}_{*, T^{(1)}} })\eta_n\rangle\\&= \langle (I_{E_1^{\otimes n}}\ot	(I_{E_2} \ot \Delta_*(T^{(1)}))\wT^{(2)^*}\big)\wT^{(1)^*}_{n}h,(t_{2,1}^{(1,n)} \otimes I_{\mathcal{D}_{*, T^{(1)}} }) \eta_n\rangle.
	\end{align*} 
	On the other hand, since  
	
	\begin{align*}&
		\langle{(I_{E_{2}} \ot\Pi_{1}) \wT^{(2)}}^{*}h, \eta_n\rangle 
		=\langle{(I_{E_{2}} \ot\sum_{k\geq0}(I_{E_{1}^{\otimes k}}\otimes \Delta_*(T^{(1)}))\wT^{(1)^*}_{k}) \wT^{(2)}}^{*}h , \eta_n\rangle \\&=
		\langle{(I_{E_{2}} \ot(I_{E_{1}^{\otimes n}}\otimes\Delta_*(T^{(1)}))\wT^{(1)^*}_{n}) \wT^{(2)}}^{*}h , \eta_n\rangle \\&=
		\langle h ,{\wT^{(2)}(I_{E_{2}} \ot\wT^{(1)}_{n}(I_{E_{1}^{\otimes n}}\otimes\Delta_*(T^{(1)}))) } \eta_n\rangle  \\&=
		\langle h ,{\wT^{(1)}_{n}(I_{E_{1}^{\otimes n}} \ot\wT^{(2)}(I_{E_{2}}\otimes\Delta_*(T^{(1)}))) }(t_{2,1}^{(1,n)} \ot I_{\mathcal{D}_{*, T^{(1)}} })\eta_n\rangle \\&=
		\langle (I_{E_1^{\otimes n}}\ot(I_{E_2} \ot \Delta_*(T^{(1)}))\wT^{(2)^*})\wT^{(1)^*}_{n}h ,(t_{2,1}^{(1,n)} \ot I_{\mathcal{D}_{*, T^{(1)}} })\eta_n\rangle  =
		\langle\widetilde{M}^{*}_{\Theta}\Pi_{1}h ,\eta_n\rangle,
	\end{align*} we get $ {(I_{E_{2}} \ot\Pi_{1}) \wT^{(2)}}^{*} = \widetilde{M}^{*}_{\Theta}\Pi_{1}.$ This complets the proof.
\end{proof} 

\begin{observation}
	\begin{enumerate}
		\item In the setting of Theorem \ref{DSAR2}, if $\mathcal{Q}:=\Pi_1\mathcal{H},$ then $$\rho(a)\mathcal{Q}\subseteq \mathcal{Q}, ~\widetilde{S}^{
			(1)^*}(\mathcal{Q})\subseteq E_1\ot \mathcal{Q}\quad \mbox{and}\quad \widetilde{M}_{\Theta}^*(\mathcal{Q})\subseteq E_2\ot \mathcal{Q} \quad \mbox{for}\quad a\in \mathcal{M}.$$
		\item The representations $(\pi',X,Y)$ and $(\sigma, T^{(1)}, T^{(2)})$ of $\mathbb{E}$ acting on the Hilbert spaces $\mathcal{Q}$ and $\mathcal{H},$ respectively, are isomorphic, where $\pi'=\rho|_{\mathcal{Q}},$
		$\widetilde{X}= P_{\mathcal{Q}}\widetilde{S}^{(1)}|_{E_1\ot \mathcal{Q}}$ and $\widetilde{Y}= P_{\mathcal{Q}}\widetilde{M}_{\Theta}|_{E_2\ot \mathcal{Q}}.$
		\item Theorem \ref{DSAR2} remains valid if we drop the assumption that $ \dim (\mathcal{D}_{*, T^{(1)}} \oplus  (E_1 \ot\mathcal{D}_{*, T^{(2)}})) < \infty$ and $ \dim ((E_2 \ot \mathcal{D}_{*, T^{(1)}}) \oplus  \mathcal{D}_{*, T^{(2)}} )<\infty.$
	\end{enumerate}
\end{observation}


\subsection*{Conflict of Interest}
We state that there is no conflict of interest and we have no personal relationships that could have appeared to influence the work reported in this paper.

\subsection*{Data Availability}
Data sharing is not applicable to this article as no datasets were generated or analyzed during the current study.

\subsection*{Acknowledgement}
The authors thank the reviewer for carefully reading the manuscript and the editor for suggesting changes in the manuscript. The authors want to thank Harsh Trivedi and Shankar Veerabathiran for some fruitful discussions.

\end{document}